\documentclass[12pt,reqno,intlimits]{amsart}
\usepackage[cp1251]{inputenc}
\usepackage[T2A]{fontenc}
\usepackage{amsfonts,amsmath,amsxtra,amsthm,amssymb,latexsym}

\theoremstyle{plain}
\newtheorem {theorem} {Theorem}[section]

\theoremstyle{definition}
\newtheorem {definition}[theorem]{Definition}
\newtheorem {remark}[theorem]{Remark}
\newtheorem {example}[theorem]{Example}

\begin{document}

\title{Conditions of fixed sign for {\normalsize $n\times n$} operator matrices}
\author{I.V. Orlov and E.V. Bozhonok}
\maketitle

\begin{abstract}
A positive definiteness criterion and, under the additional
conditions, a nonnegativity criterion for a self-adjoint
continuous operator matrix, acting in product of an arbitrary
number of real separable Hilbert spaces, are obtained. As
application, both sufficient and necessary analytical conditions
of functional extremum of several Hilbert variables are
considered.
\\\textit{Mathematics subject classification} (2000): 47A63, 47A07.
\\\textit{Key words and phrases}: Operator matrix, Hilbert space,
positive definiteness criterion, nonnegativity criterion, Schur
complements.
\end{abstract}

\section{\textbf{Introduction}}

The conditions of fixed sign for the operator matrices acting in a
product of Hilbert spaces play a significant role in operator
theory and its numerous applications. Among the base results
respective to ${2\times 2}$ matrix, note the results by M.G. Krein
and Ju.L. Shmulian~(\cite{1},~\cite{2}). In the various applied
aspects, the conditions of fixed sign for ${2\times 2}$ operator
matrices arose in the works by V.P. Potapov~\cite{4}, T.Ya. Azizov
and I.S. Iohvidov~\cite{11}, N.D. Kopachevskii and S.G.
Krein~(\cite{40}--\cite{10}), J. Mandel~\cite{7}. The specific
direction, Schur analysis~\cite{5} using operator Schur
complements, appeared later (see survey by A. Bultheel and K.
M\"{u}ller~\cite{6}), the paper of Y. Arlinskii~\cite{13}).

Among the recent works in this area, distinguish the results of S.
Hassi, M. Malamud, H. de Snoo~\cite{123} and T.Ya. Azizov and V.A.
Khatskevich~\cite{8}.

However, the most of the works above is connected with ${2\times
2}$ operator matrices. In the present work we are interested in a
complete description of the positive definite and nonnegative
operator matrices, including the suitable explicit conditions for
${3\times 3}$ operator matrices.

In the paper, the self-adjoint linear continuous operator matrices
acting in a product of an arbitrary number of the real separable
Hilbert spaces are studied. In the first item, a positive
definiteness criterion in terms of the first kind Schur operators
(Theorem~\ref{Th 8.2}) is obtained and the precise number of the
corresponding inequalities is calculated (Theorem~\ref{Th 8.2.2}).
In the second item, under the additional conditions, a
nonnegativity criterion in term of the second kind Schur operators
(Theorem~\ref{Th 8.4}) is obtained. In the third item,
applications of the results above to determination of the
functional extrema in a product of several Hilbert spaces are
investigated.

Throughout the paper, the case ${n=3}$ is described explicitly
(Examples~\ref{Exam. 1} and~\ref{Exam. 1.2}, Theorem~\ref{th.2.12}
and~\ref{th.2.12.1}). In the conclusion, an example of
determination of functional extremum in the case of ${n=3}$ and
non-commuting second partial derivatives is considered.

\section{\textbf{Positive definiteness criterion
for self-adjoint operator matrix in product of
$\textit{\textbf{n}}$ Hilbert spaces}}\label{s12.1}

\quad

We are based on the well known criterion for $2\times 2$ matrix by
means of Schur complements~\cite{11}.

\begin{theorem}\label{oiv.lem.2.8}
Let $H_{1}$, $H_{2}$ be separable real Hilbert spaces,
${B_2=(B_{ij}:H_{j}\rightarrow H_{i})_{i,j=1}^2}$ is a linear
self-adjoint continuous operator in ${H_1\times H_2}$
(${B_{11}=B_{11}^*}$, ${B_{22}=B_{22}^*}$, ${B_{12}=B_{21}^*}$).
Then $B_{2}$ is positive definite (${B_2\gg 0}$) if and only if
the following conditions are fulfilled:
\begin{enumerate}
\item [i)] ${B_{11}\gg 0}$, ${B_{22}\gg 0}$;

\item [ii)] ${\bigtriangleup_{2}^{1}(B_2):=B_{11}-B_{12}\cdot
  B_{22}^{-1}\cdot B_{21}\gg 0}$,
  ${\bigtriangleup_{1}^{2}(B_2):=B_{22}-B_{21}\cdot
  B_{11}^{-1}\cdot B_{12}\gg 0}$.
\end{enumerate}
\end{theorem}

\begin{remark}
It's easy to check the following relation between Schur
complements:
$$
B_{21}B^{-1}_{11}\bigtriangleup_{2}^{1}=
\bigtriangleup_{1}^{2}B^{-1}_{22}B_{21}\,\,\,\,\mbox{or}\,\,\,\,
B_{12}B^{-1}_{22}\bigtriangleup_{1}^{2}=
\bigtriangleup_{2}^{1}B^{-1}_{11}B_{12}
$$
\end{remark}

Let's consider now a linear continuous operator $B_{n}$ acting in
product of $n$ separable real Hilbert spaces ${H=H_{1} \times
\ldots \times H_{n}}$ and defined by operator matrix
${(B_{ij})^{n}_{i,j=1}}$. First, introduce \textit{Schur operators
of the first kind}.

\begin{definition}\label{Defin 8.1}
Let's divide the matrix $B_{n}$ to four units:
 $B_{n}^{11}$ is main minor of the order ${\left[\frac{n}{2}\right]\times
 \left[\frac{n}{2}\right]}$, $B_{n}^{22}$ is its adjacent minor of the order
 ${\left(n-\left[\frac{n}{2}\right]\right)\times
 \left(n-\left[\frac{n}{2}\right]\right)}$,
 $B_{n}^{12}$ and $B_{n}^{21}$ are corresponding rectangular blocks
 of the orders ${\left[\frac{n}{2}\right]\times
 \left(n-\left[\frac{n}{2}\right]\right)}$ and
 ${\left(n-\left[\frac{n}{2}\right]\right)\times \left[\frac{n}{2}\right]}$, respectively.
 (Here ${[\,\cdot\,]}$ denotes integral part of a number). Suppose
 that the block matrix ${(B^{ij}_n)^{2}_{i,j=1}}$ satisfies
necessary condition of positive definiteness $i)$ from
Theorem~\ref{oiv.lem.2.8} and hence $B_{n}^{11}$ $B_{n}^{22}$ are
continuously invertible.

Let's introduce now four types of {\it Schur operators of the
first kind}:
 $$
  \bigtriangleup_{1}^{1}(B_{n})=B^{11}_{n}; \quad
  \bigtriangleup_{2}^{1}(B_{n})=B^{11}_{n}-B^{12}_{n}\cdot
  (B^{22}_{n})^{-1}\cdot B^{21}_{n};
 $$
 $$
  \bigtriangleup_{2}^{2}(B_{n})=B^{22}_{n}; \quad
  \bigtriangleup_{1}^{2}(B_{n})=B^{22}_{n}-B^{21}_{n}\cdot
  (B^{11}_{n})^{-1}\cdot B^{12}_{n}.
 $$

\end{definition}

Note that the matrices ${\bigtriangleup_{j}^{i}(B_{n})}$
 have maximal order ${\left(\left[\frac{n}{2}\right]+1\right)\times
 \left(\left[\frac{n}{2}\right]+1\right)}$ in case of the odd $n$ and
 ${\left[\frac{n}{2}\right]\times
 \left[\frac{n}{2}\right]}$ in case of the even $n$. Pass to the
 main result.

\begin{theorem}\label{Th 8.2}
Let $H_{i}$ ${(i=\overline{1,n})}$ be separable real Hilbert
spaces, ${H=H_{1} \times \ldots \times H_{n}}$,
 ${B_{n}=(B_{ij})_{i,j=1}^{n}}$ be a linear continuous self-adjoint operator in
 $H$, in the ${B_{ij}\in \mathcal{L}(H_{j},H_{i})}$; ${(i,j=\overline{1,n})}$.
 Then $B_{n}$ is
positive definite if and only if the following system of positive
definiteness inequalities is fulfilled:
 \begin{equation}\label{boz_e8.1}
  \{\bigtriangleup^{i_{m}}_{j_{m}}\ldots
  \bigtriangleup^{i_{2}}_{j_{2}}\bigtriangleup^{i_{1}}_{j_{1}}(B_{n})\gg
  0\} \quad (i_{l},j_{l}=1, 2),
 \end{equation}
where
 \begin{equation}\label{boz_e8.2}
  m=\left\{
  \begin{array}{ccc}
   k, & \mbox {as} & n=2^{k}\\
   k+1, & \mbox {as} & 2^{k}<n<2^{k+1}.\\
  \end{array}
  \right.
 \end{equation}
\end{theorem}

\begin{proof} Let's divide the matrix ${B_{n}=(B_{ij})_{i,j=1}^n}$
to four units, according to Definition~\ref{Defin 8.1},
${B_{n}=(B^{ij}_{n})_{i,j=1}^2}$. Let ${\tilde{H}^{1}=H_{1} \times
\ldots \times H_{[\frac{n}{2}]}}$,
${\tilde{H}^{2}=H_{[\frac{n}{2}]+1} \times ... \times H_{n}}$.
Then it can be consider $B_{n}$ as an operator matrix acting in
${\tilde{H}^{1}\times \tilde{H}^{2}}$, in there ${B^{ij}_{n}\in
\mathcal{L}(\tilde{H}^{j}\rightarrow \tilde{H}^{i})}$: $i,j=1,2$.
By virtue of Theorem~\ref{oiv.lem.2.8}, $B_{n}$ is positive
definite if and only if the following inequalities
$$
1) B^{11}_{n}=\bigtriangleup_{1}^{1}(B_{n})\gg 0; \quad
B^{22}_{n}=\bigtriangleup_{2}^{2}(B_{n})\gg 0;
$$
$$
2) \bigtriangleup_{2}^{1}(B_{n}) \gg 0; \quad
\bigtriangleup_{1}^{2}(B_{n}) \gg 0
$$
hold true. Again, dividing every from the operator matrices above
to four units, according to Definition~\ref{Defin 8.1}, and
applying Theorem~\ref{oiv.lem.2.8} for the corresponding operators
acting in ${\tilde{H}^{1}\times\tilde{H}^{1}}$ and
${\tilde{H}^{2}\times\tilde{H}^{2}}$, respectively, it follows 16
inequalities of the form
${\{\bigtriangleup^{i_{2}}_{j_{2}}\bigtriangleup^{i_{1}}_{j_{1}}(B_n)\gg
0\}}$; ${(i_{l},j_{l}=1,2)}$ for the spaces ${\left(H_{1} \times
\ldots \times
H_{\bigl[\frac{[\frac{n}{2}]}{2}\bigr]}\right)^{2},\,}$ $
{\left(H_{\bigl[\frac{[\frac{n}{2}]}{2}\bigr]+1} \times \ldots
\times H_{[\frac{n}{2}]}\right)^{2},\,\left(H_{[\frac{n}{2}]+1}
\times \ldots \times
H_{\bigl[\frac{n-[\frac{n}{2}]}{2}\bigr]}\right)^{2}}$ and $
{\left(H_{\bigl[\frac{n-[\frac{n}{2}]}{2}\bigr]+1} \times \ldots
\times H_{n}\right)^{2}}$ respectively. Prolonging by the
construction and reducing the matrix sizes at least up to
${2^{k+1-p}\times 2^{k+1-p}}$ as ${n<2^{k+1}}$ under $p$-th
inductive step, it leads after $m$ inductive steps to the
system~\eqref{boz_e8.1}.
\end{proof}

\begin{remark}
1) It's easy to see that in case of the commuting operators
$B_{ij}$, the system~\eqref{boz_e8.1} can be reduced to the well
known Sylvester condition of positive definiteness of $n$ main
minors of the matrix $B_n$.

2) Denoting by ${[x]_+}$ the "right integral part" of $x$ (i.e.
the nearest integer to $x$ from the right), it can rewrite the
condition~\eqref{boz_e8.2} in the form
$$
m=[\log_2 n]_+\,\,.
$$
\end{remark}

Let's estimate the number of inequalities in the
system~\eqref{boz_e8.1}.
\begin{theorem}\label{Th 8.2.2}
 The number of inequalities $V_{n}$ in system~\eqref{boz_e8.1} is
 \begin{equation}\label{boz_e8.3}
  V_{n}=2^{k}\cdot(3n-2^{k+1})\quad\mbox{as}\quad 2^{k}\leq n \leq
  2^{k+1}\quad (k\in\mathbb{N}_0,\,n\in\mathbb{N}).
 \end{equation}
 Moreover,
 \begin{equation}\label{boz_e8.4}
  n^{2}\leq V_{n} \leq (n+1)^{2}.
 \end{equation}
\end{theorem}

\begin{proof}
From definition of the system~\eqref{boz_e8.1}, the following
recursion relations
$$
V_{1}=1;\,\,V_{2n}=4V_{n};\,\,V_{2n+1}=2(V_{n}+V_{n+1})\quad(n\in\mathbb{N})
$$
easily follow. Consider three possible cases.
\begin{enumerate}
\item [1)] $2^k\leq n<2^{k+1}$, ${n=2m}$ $(m\in\mathbb{N})$. Then
${2^{k-1}\leq m < 2^{k}}$ and applying~\eqref{boz_e8.3}, it
implies
$$
V_{n}=4V_{m}=4\cdot 2^{k-1}\cdot(3m-2^{k})= 2^{k}\cdot(3\cdot
2m-2^{k+1})=2^{k}\cdot(3n-2^{k+1}).
$$

\item [2)] $2^k\leq n<2^{k+1}$, ${n=2m+1}$ $(m\in\mathbb{N_0})$.
Then ${2^{k}\leq 2m+1 < 2^{k+1}}$, whence ${2^{k-1}< m+1 \leq
2^{k}}$ and applying~\eqref{boz_e8.3}, it follows
$$
V_{n}=2(V_{m}+V_{m+1})=2^{k}\cdot(3n-2^{k+1}).
$$
\item [3)] ${n=2^{k+1}}$ $(k\in\mathbb{N_0})$. Then
$$
V_{n}=V_{2^{k+1}}=4V_{2^{k}}=4\cdot2^{k-1}\cdot(3\cdot2^{k}-2^{k})
=2^{k}\cdot(3n-2^{k+1}).
$$
\end{enumerate}
Thus, the equality~\eqref{boz_e8.3} is obtained. In addition,
${V_{n}=n^{2}}$ as ${n=2^{k}}$ and ${V_{n}=(n+1)^{2}}$ as
${n=2^{k+1}}$. Whence, in view of increase $V_{n}$ for ${2^{k}\leq
n \leq 2^{k+1}}$, the inequality~\eqref{boz_e8.4} follows.
\end{proof}

Pass to the examples. First, consider the partial case ${n=3}$.

\begin{example}\label{Exam. 1}
Introduce (in case of ${B_3}$ with invertible $B_{ii}$
${(i=1,2,3)}$) generalized Schur complements
$$
\bigtriangleup_{jk}^{i}=B_{ii}-B_{ij}\cdot B_{jj}^{-1}\cdot
B_{jk}\quad (i,j,k=1,2,3).
$$
Then direct calculations show that the
inequalities~\eqref{boz_e8.1} take form of
$$
B_{11}\gg 0; \quad B_{22}\gg 0; \quad B_{33}\gg 0;
$$
$$
\bigtriangleup_{12}^{2}\gg 0; \quad \bigtriangleup_{32}^{2}\gg 0;
\quad \bigtriangleup_{13}^{3}\gg 0; \quad
\bigtriangleup_{23}^{3}\gg 0;
$$
\begin{equation}\label{boz_e8.10}
\bigtriangleup_{12}^{2}-\bigtriangleup_{13}^{2}\cdot
(\bigtriangleup_{13}^{3})^{-1}\cdot\bigtriangleup_{12}^{3}\gg 0;
\quad \bigtriangleup_{13}^{3}-\bigtriangleup_{12}^{3}\cdot
(\bigtriangleup_{12}^{2})^{-1}\cdot\bigtriangleup_{13}^{2}\gg 0;
\end{equation}
$$
B_{11}-(B_{12}\cdot B_{22}^-\cdot B_{21}+B_{12}\cdot B_{23}^-\cdot
B_{31}+B_{13}\cdot B_{32}^-\cdot B_{21}+B_{13}\cdot B_{33}^-\cdot
B_{31})\gg 0.
$$
Here we denote by $B_{ij}^-$ corresponding elements of the inverse
matrix to ${(B_{ij})_{i,j=2}^3}$. In this case,
$$
B_{22}^-=B_{22}^{-1}\cdot (I_{H_2}+B_{23}\cdot
(\bigtriangleup_{23}^{3})^{-1}\cdot B_{32} \cdot
B_{22}^{-1})=(I_{H_2}+B_{22}^{-1}\cdot B_{23}\cdot
(\bigtriangleup_{23}^{3})^{-1}\cdot B_{32})\cdot B_{22}^{-1};
$$
$$
B_{33}^-=B_{33}^{-1}\cdot (I_{H_3}+B_{32}\cdot
(\bigtriangleup_{32}^{2})^{-1}\cdot B_{23} \cdot
B_{33}^{-1})=(I_{H_3}+B_{33}^{-1}\cdot B_{32}\cdot
(\bigtriangleup_{32}^{2})^{-1}\cdot B_{23})\cdot B_{33}^{-1};
$$
$$
B_{23}^-=-(\bigtriangleup_{32}^{2})^{-1}\cdot B_{23} \cdot
B_{33}^{-1}=-B_{22}^{-1}\cdot
B_{23}(\bigtriangleup_{23}^{3})^{-1};
$$
$$
B_{32}^-=-(\bigtriangleup_{23}^{3})^{-1}\cdot B_{32} \cdot
B_{22}^{-1}=-B_{33}^{-1}\cdot
B_{32}(\bigtriangleup_{32}^{2})^{-1}.
$$
Note also that the expressions in the third row
of~\eqref{boz_e8.10} can be considered as the "second order Schur
complements", $\bigtriangleup_{123}^{23}$ and
$\bigtriangleup_{132}^{32}$, relatively.
\end{example}

Further, consider the case of "bidiagonal" operator matrix $B_n$.
In this case general definiteness conditions~\eqref{boz_e8.1}
assume an essential simplification.

\begin{example}\label{Exam. 2}
Let $B_{n}=\{B_{ij}\}_{i,j=1}^{n}$ be operator matrix in $H$ with
nonzero principal diagonals and zero all other elements.

1) In case ${n=2m}$,
$$
B_{n}= \left(\begin{array}{cccccccc} B_{11} & 0 & ... & 0 & 0 &
... & 0 &
B_{1n}\\
0 & B_{22} & \ldots & 0 & 0 & \ldots & B_{2,n-1} & 0\\
\vdots & \vdots &  & \vdots & \vdots &  & \vdots & \vdots\\
 0 & 0 & \ldots & B_{\frac{n}{2},\frac{n}{2}} & B_{\frac{n}{2},\frac{n}{2}+1} & \ldots & 0 &
 0\\
0 & 0 & \ldots & B_{\frac{n}{2}+1,\frac{n}{2}} &
B_{\frac{n}{2}+1,\frac{n}{2}+1} & \ldots & 0 &
 0\\
\vdots & \vdots &  & \vdots & \vdots &  & \vdots & \vdots\\
B_{n1} & 0 & \ldots & 0 & 0 & \ldots & 0 & B_{nn}
 \end{array} \right).
$$
Then the units ${B^{ij}_{n}}$ take form
$$
B_{n}^{11}=
\left(\begin{array}{ccc} B_{11} &  ... & 0\\
\vdots &  & \vdots\\
 0 &  \ldots & B_{\frac{n}{2},\frac{n}{2}}
 \end{array} \right), \quad B_{n}^{22}= \left(\begin{array}{ccc}
B_{\frac{n}{2}+1,\frac{n}{2}+1} & \ldots &
 0\\
\vdots &  &\vdots\\
 0 & \ldots  & B_{nn}
 \end{array} \right),
 $$
 $$
 B_{n}^{12}= \left(\begin{array}{ccc}
0 & \ldots  & B_{1n}\\
\vdots &  & \vdots\\
 B_{\frac{n}{2},\frac{n}{2}+1} & \ldots & 0
 \end{array} \right), B_{n}^{21}= \left(\begin{array}{ccc}
0 & \ldots & B_{\frac{n}{2}+1,\frac{n}{2}}\\
\vdots &  & \vdots\\
 B_{n1} & \ldots & 0
 \end{array} \right).
 $$
 Applying Schur operators, we obtain the diagonal matrix
 $$
 \bigtriangleup_{1}^{1}(B_{n})=B^{11}_{n},\,\,
 diag \bigtriangleup_{2}^{1}(B_{n})=\{B_{ii}-B_{i,n+1-i}\cdot B^{-1}_{n+1-i,n+1-i}\cdot B_{n+1-i,i}\}_{i=1}^{\frac{n}{2}}
 $$
$$
\bigtriangleup_{2}^{2}(B_{n})=B^{22}_{n},\,\,
 diag \bigtriangleup_{1}^{2}(B_{n})=\{B_{ii}-B_{i,n+1-i}\cdot B^{-1}_{n+1-i,n+1-i}\cdot B_{n+1-i,i}\}_{i=\frac{n}{2}+1}^{n}
 $$
Since conditions~\eqref{boz_e8.1}--\eqref{boz_e8.2} for a diagonal
matrix take form $ B_{ii}\gg0$ then ${B_{n}\gg0}$ if and only if
the inequalities
$$ B_{ii}\gg0; \quad B_{ii}-B_{i,n+1-i}\cdot
B^{-1}_{n+1-i,n+1-i}\cdot
B_{n+1-i,i}\gg0\,\,\quad(i=\overline{1,n})
$$
hold.

 2) In case ${n=2m+1}$
 $$
 B_{n}= \left(\begin{array}{ccccccccc}
B_{11} & 0 & \ldots & 0 & 0 & 0 & \ldots & 0 &
B_{1n}\\
0 & B_{22} & \ldots & 0 & 0 & 0 & \ldots & B_{2,n-1} & 0\\
\vdots & \vdots &  & \vdots & \vdots & \vdots &  & \vdots & \vdots\\
 0 & 0 & \ldots & B_{\frac{n-1}{2},\frac{n-1}{2}}& 0 & B_{\frac{n-1}{2},\frac{n+3}{2}} & \ldots & 0 &
 0\\
0 & 0 & \ldots & 0 & B_{\frac{n+1}{2},\frac{n+1}{2}} & 0 & \ldots
& 0 &
 0\\
\vdots & \vdots &  & \vdots & \vdots & \vdots &  & \vdots & \vdots\\
B_{n1} & 0 & \ldots & 0 & 0 & 0 & \ldots & 0 & B_{nn}
 \end{array} \right).
 $$
The calculations, analogous with the previous case, lead to the
following criterion: $B_{n}\gg 0$ if and only if the inequalities
$$
B_{ii}\gg0\,\,(i=\overline{1,n}); \quad B_{ii}-B_{i,n+1-i}\cdot
B^{-1}_{n+1-i,n+1-i}\cdot B_{n+1-i,i}\gg0\,\,\,
(i=\overline{1,n},\,i\neq \frac{n+1}{2})
$$
hold.
\end{example}

Finally, Theorem~\ref{Th 8.2} can be formulated as positive
definiteness criterion for the quadratic forms acting in a product
of Hilbert spaces.

\begin{theorem}\label{Th 8.5}
Let $H_{i}$ ${(i=\overline{1,n})}$ be separable real Hilbert
spaces, ${H=H_{1} \times ... \times H_{n}}$, $\varphi$ is
continuous quadratic form on $H$ associated with a self-adjoint
linear continuous operator ${B_{n}=(B_{ij})_{i,j=1}^n}$, i.e.
${\varphi(h)=\langle B_{n}h,h \rangle}$. Then the form $\varphi$
is positive definite on $H$ if and only if the
conditions~\eqref{boz_e8.1}--\eqref{boz_e8.2} are fulfilled.
\end{theorem}


\section{\textbf{Nonnegativity conditions
for self-adjoint operator matrix in product of
$\textit{\textbf{n}}$ Hilbert spaces}}\label{s12.2}

\quad

First, consider a case of $2\times 2$ operator matrix
${B_{2}=(B_{ij})_{i,j=1}^2}$ acting in ${H_{1}\times H_{2}}$. For
the case of product of the two complex Hilbert spaces a general
nonnegativity criterion for $B_{2}$ was obtained
in~(\cite{123},~cor. 2.2) in terms of linear relations. In this
connection, self-adjointness of $B_{2}$ is necessary condition for
$B_2\gg 0$. We formulate below for real case a simplified
sufficient condition of nonnegativity $B_{2}$ under simplifying
assumptions of a "partial self-adjointness" of $B_{2}$ and
continuous invertibility of $B_{11}$.

\begin{theorem}\label{Th 8.3}
Let $H_{1}$, $H_{2}$ be real separable Hilbert spaces,
${B_{2}=(B_{ij})_{i,j=1}^2}$ is continuous linear operator in
${H_{1}\times H_{2}}$, where $B_{11}$ is continuously invertible
and self-adjoint and ${B_{12}=B_{21}^{*}}$. Then $B_{2}$ is
nonnegative if and only if
$$
B_{11}\geq 0\quad \mbox{and}\quad
\bigtriangleup_{1}^{2}(B_{2})=B_{22}-B_{21}\cdot
(B_{11})^{-1}\cdot B_{12}\geq 0;
$$
or, that is equivalent,
$$
B_{22}\geq 0\quad \mbox{and}\quad
\bigtriangleup_{2}^{1}(B_{2})=B_{11}-B_{12}\cdot
(B_{22})^{-1}\cdot B_{21}\geq 0.
$$
\end{theorem}

The proof follows from the verified identity
$$
\langle B_{2}h,h\rangle=
\|B_{11}^{1/2}h_{1}+B_{11}^{-1/2}B_{12}h_{2}\|^{2}+\langle
\widetilde{\bigtriangleup_{2}^{1}}h_{2},h_{2}\rangle,\,\,\,\,\,\,
h=(h_1,h_{2})\in H_{1}\times H_{2},
$$
see, e.g.~(\cite{9}, Vol. I, Ch. I, 1.3). Note also that, in
accordance with~(\cite{11}, Ch. II, Lemma 3.21), invertibility
condition for $B_{11}$ can be replaced by the considerably more
general condition of invertibility of
$B_{11}\bigl|\bigr._{\overline{ran\, B_{11}}}$\,\,\,.

Let's pass to general case. Once again, consider now a linear
continuous operator $B_{n}$ acting in product of $n$ separable
real Hilbert spaces ${H=H_{1} \times \ldots \times H_{n}}$ and
defined by operator matrix ${(B_{ij})^{n}_{i,j=1}}$. Introduce
\textit{Schur operators of the second kind}.

\begin{definition}\label{Defin 8.2}
 Let's divide the matrix $B_{n}$ to four units:
 ${\widetilde{B}_{n}^{11}=B_{11}}$,
 $\widetilde{B}_{n}^{22}$ is its adjacent minor of
the order ${(n-1)\times (n-1)}$, $\widetilde{B}_{n}^{12},
 \widetilde{B}_{n}^{21}$ are corresponding row and
column of the matrix $B_n$, having size ${1\times (n-1)}$ and
${(n-1)\times 1}$, respectively. Suppose that
${\widetilde{B}_{n}^{11}=B_{11}}$ is continuously invertible.
Let's introduce now two types of {\it Schur operators of the
second kind}:
$$
\widetilde{\bigtriangleup_{1}^{1}}(B_{n})=\widetilde{B}^{11}_{n};
\quad
\widetilde{\bigtriangleup_{1}^{2}}(B_{n})=\widetilde{B}^{22}_{n}-\widetilde{B}^{21}_{n}\cdot
(\widetilde{B}^{11}_{n})^{-1}\cdot \widetilde{B}^{12}_{n}.
$$
\end{definition}

Applying repeatedly second kind Schur operators to
$B_{n}=(B_{ij})_{i,j=1}^n$ it is not difficult, by analogy with
the proof of Theorem~\ref{oiv.lem.2.8}, to generalize
Theorem~\ref{Th 8.3} for the case of an arbitrary $n$.

\begin{theorem}\label{Th 8.4}
Let $H_{i}$ ${(i=\overline{1,n})}$ be separable real Hilbert
spaces, ${H=H_{1} \times \ldots \times H_{n}}$,
$B_{n}=(B_{ij})_{i,j=1}^n$ be a linear continuous self-adjoint
operator in $H$, in there ${B_{ij}\in \mathcal{L}(H_{j},H_{i})}$;
 ${i,j=\overline{1,n}}$. Suppose moreover that the all operators
 \begin{equation}\label{boz_e8.5}
  \widetilde{\bigtriangleup_{1}^{1}}(\widetilde{\bigtriangleup_{1}^{2}})^{k}(B_{n})\quad
  (k=\overline{0,n-2}),
 \end{equation}
 are continuously invertible. Then $B_{n}$ is nonnegative if and
 only if the following system of nonnegativity inequalities is
 fulfilled:
 \begin{equation}\label{boz_e8.6}
  \widetilde{\bigtriangleup_{1}^{1}}(\widetilde{\bigtriangleup_{1}^{2}})^{k}(B_{n})\geq
  0\quad (k=\overline{0,n-2});\quad
  (\widetilde{\bigtriangleup_{1}^{2}})^{n-1}(B_{n})\geq 0.
 \end{equation}
\end{theorem}

\begin{proof}
Let's divide the matrix ${B_{n}=(B_{ij})^{n}_{i,j=1}}$ to four
units, according to Definition~\ref{Defin 8.2},
${B_{n}=(\widetilde{B}_{n}^{ij})_{i,j=1}^2}$. Let
${\tilde{H}^{1}=H_{1}}$, ${\tilde{H}^{2}=H_{2} \times \ldots
\times H_{n}}$. Then it can consider $B_{n}$ as an operator matrix
acting in ${\tilde{H}^{1} \times \tilde{H}^{2}}$, in there
${\widetilde{B}_{n}^{ij}\in \mathcal{L}(\tilde{H}^{j},
\tilde{H}^{i})}$; ${i,j=1,2}$. Under assumptions of continuous
invertibility of
${\widetilde{B}_{n}^{11}=B_{11}=\widetilde{\bigtriangleup_{1}^{1}}(B_{n})}$
and self-adjointness of $B_{n}$, by virtue of Theorem~\ref{Th
8.3}, $B_{n}$ is nonnegative if and only if the following
inequalities
$$
\widetilde{\bigtriangleup_{1}^{1}}(B_{n})\geq0,\qquad
\widetilde{\bigtriangleup_{1}^{2}}(B_{n})=\tilde{B}^{22}_{n}-
\tilde{B}^{21}_{n}\cdot(\tilde{B}^{11}_{n})^{-1}\cdot\tilde{B}^{12}_{n}\geq0
$$
hold true. Again, dividing the operator matrix
${\widetilde{\bigtriangleup_{1}^{2}}(B_{n})}$ of the size
${(n-1)\times(n-1)}$ according to Definition~\ref{Defin 8.2}, it
can apply Theorem~\ref{Th 8.3} now to the operator
${\widetilde{\bigtriangleup_{1}^{2}}(B_{n})}$. Here
self-adjointness $B_{n}$ implies one for
${\widetilde{\bigtriangleup_{1}^{2}}(B_{n})}$ and continuous
invertibility of
${\widetilde{\bigtriangleup_{1}^{1}}(\widetilde{\bigtriangleup_{1}^{2}}(B_{n}))}$
follows from~\eqref{boz_e8.5}. Hence, we obtain the following
necessary and sufficient nonnegativity conditions for
${\widetilde{\bigtriangleup_{1}^{2}}(B_{n})}$:
$$
\widetilde{\bigtriangleup_{1}^{1}}\widetilde{\bigtriangleup_{1}^{2}}(B_{n})\geq0,\quad
(\widetilde{\bigtriangleup_{1}^{2}})^{2}(B_{n})\geq0.
$$
Prolonging by induction the construction and reducing the matrix
size up to ${(n-p)\times(n-p)}$ under the $p$-th inductive steps,
it leads after $n$ inductive steps to the system~\eqref{boz_e8.6}.
\end{proof}

\begin{remark}
Note that general system of conditions in Theorem~\ref{Th 8.4}
consists of $n$ nonnegativity inequalities~\eqref{boz_e8.6} and
${n-1}$ invertibility conditions~\eqref{boz_e8.5}.
\end{remark}

Finally, Theorem~\ref{Th 8.4} can be formulated as nonnegativity
condition for the quadratic forms acting in a product of Hilbert
spaces.

\begin{theorem}\label{Th 8.6}
Let $H_{i}$ ${(i=\overline{1,n})}$ be separable real Hilbert
spaces, ${H=H_{1} \times ... \times H_{n}}$, $\varphi$ is
continuous quadratic form on $H$ associated with a self-adjoint
linear continuous operator ${B_{n}=(B_{ij})_{i,j=1}^n}$, i.e.
${\varphi(h)=\langle B_{n}h,h \rangle}$ and the
conditions~\eqref{boz_e8.5} are fulfilled. Then the form $\varphi$
is nonnegative on $H$ if and only if the
inequalities~\eqref{boz_e8.6} hold true.
\end{theorem}

Finally, consider in this item also the partial case ${n=3}$.

\begin{example}\label{Exam. 1.2}
Introduce (in case of ${B_3}$ with an invertible $B_{jj}$)
generalized second order Schur complements
$$
\bigtriangleup^{ik}_{j}=B_{ik}-B_{ij}\cdot B_{jj}^{-1}\cdot
B_{jk}\quad (i,j,k=1,2,3).
$$
Then direct calculations show that the invertibility
conditions~\eqref{boz_e8.5} take form of invertibility of the
operator
$$
B_{11}\quad \mbox{and} \quad
\bigtriangleup_{1}^{22}=B_{22}-B_{21}\cdot(B_{11})^{-1}\cdot
B_{12},
$$
and the nonnegativity conditions~\eqref{boz_e8.6} take form of
$$
B_{11}\gg 0; \quad \bigtriangleup_{1}^{22}\gg 0; \quad \mbox{and}
\quad \bigtriangleup_{1}^{33}-\bigtriangleup_{1}^{32}\cdot
(\bigtriangleup_{1}^{22})^{-1}\cdot\bigtriangleup_{1}^{23}\gg 0.
$$
\end{example}

\section{\textbf{Application to functional extrema in product of Hilbert spaces}}\label{s12.1}

\quad Remind a well known~(\cite{38},~Ch. I,~Th. 8.3.3) classical
sufficient condition of a local extremum of functional acting in
Hilbert space.
\begin{theorem}\label{oiv.th.2.7}
Let $H$ be a real Hilbert space and a functional
${\Phi:H\rightarrow\mathbb{R}}$ is twice Fr\'{e}chet
differentiable at a point $y\in H$. Suppose that ${\Phi'(y)=0}$
and ${\Phi''(y)\gg 0}$ (${\Phi''(y)\ll 0}$, respectively). Then
$\Phi$ has a strong local minimum (strong local maximum,
respectively) at~$y$.
\end{theorem}

Since, in view of Young theorem~(\cite{38},~Ch. I,~Th. 5.1.1),
bilinear form $\Phi''(y)$ is symmetric then it is associated with
self-adjoint linear operator, namely Hessian
${\mathcal{H}_n(\Phi)=(\partial_{ij}\Phi(y))_{i,j=1}^{n}}$.
Applying to $\mathcal{H}$ the positive definiteness criterion
(Theorem~\ref{Th 8.5}) together with Theorem~\ref{oiv.th.2.7}
immediately leads to the following sufficient extrema condition.

\begin{theorem}\label{Th 8.8}
Let $H_{i}$ ${(i=\overline{1,n})}$ be real separable Hilbert
spaces, ${H=H_{1} \times \ldots \times H_{n}}$, and a functional
${\Phi:H\rightarrow\mathbb{R}}$ is twice Fr\'{e}chet
differentiable at a point ${y=(y_{1},\ldots,y_{n})\in H}$. Suppose
that
\begin{enumerate}
 \item [i)] ${\partial_{1} \Phi(y)=\partial_{2} \Phi(y)=\ldots=\partial_{n} \Phi(y)=0}$;

 \item [ii)] ${\bigtriangleup^{i_{m}}_{j_{m}}\ldots
 \bigtriangleup^{i_{2}}_{j_{2}}\bigtriangleup^{i_{1}}_{j_{1}}\mathcal{H}_{n}(\Phi)(y)\gg
 0}$ \quad for ${i_{l},j_{l}=1,2}$ \,\,${(l=\overline{1,m})}$, \,\,\,${m=[\log_2
 n]_+}$\,\,.
\end{enumerate}
Then $\Phi$ has a strong local minimum at $y$.
\end{theorem}

Remind that by~\eqref{boz_e8.3}, the precise number of
inequalities in the condition (ii) is
${V_{n}=2^{k}\cdot(3n-2^{k+1})}$ as ${2^{k}\leq n \leq 2^{k+1}}$
and ${n^{2}\leq V_{n} \leq (n+1)^{2}}$. Let's distinguish the
important in practice cases of two and three variables.

\begin{theorem}\label{oiv.th.2.11}
Let a functional ${\Phi:H_1\times H_2\rightarrow\mathbb{R}}$ be
twice Fr\'{e}chet differentiable at a point ${(y_1, y_2)\in
H_1\times H_2}$. Suppose that
\begin{enumerate}
\item [i)] ${\partial_{1}\Phi(y)=0; \quad\partial_{2}\Phi(y)=0}$;

\item [ii)] ${\partial_{11}\Phi(y)\gg 0}$;
           ${\partial_{22}\Phi(y)\gg 0}$;

\item [iii)]
${\bigtriangleup_{2}^{1}=(\partial_{11}\Phi-\partial_{12}\Phi\cdot
           \partial_{22}^{-1}\Phi\cdot \partial_{21}\Phi)(y)\gg
           0}$;
           \\ ${\bigtriangleup_{1}^{2}=(\partial_{22}\Phi-\partial_{21}\Phi\cdot
           \partial_{11}^{-1}\Phi\cdot \partial_{12}\Phi)(y)\gg 0}$.
\end{enumerate}
Then $\Phi$ has a strong local minimum at ${y}$.
\end{theorem}

\begin{remark}
Note, first of all that, in case of the commuting second partial
derivatives $\partial_{ij}\Phi(y)$, conditions (ii)-(iii) take
classical form
$$
\partial_{11}\Phi(y)\gg 0;\quad (\partial_{11}\Phi\cdot
           \partial_{22}\Phi-\partial_{12}\Phi\cdot \partial_{21}\Phi)(y)\gg
           0,
$$
to within substitution ${1\leftrightarrow 2}$.

Next, in case of a strong local maximum, the \textit{all signs} in
the inequalities from (ii)-(iii) must be replaced.

At last, in~\cite{137} the result of Theorem~\ref{oiv.th.2.11} was
used to obtain sufficient conditions of the \textit{compact
extremum} (see def. in~(\cite{137},~Def. 2.1)) of variational
functional
$$
\Phi(y)=\int\limits_{a}^{b}f(x,y,y')dx\quad (y\in W_{2}^{1}
([a;b],H)).
$$
In the case of extremal ${y(\cdot)\in W_{2}^{2}}$ and ${f\in C^2
([a;b]\times H \times H)}$ these sufficient conditions of a strong
compact minimum take form
$$
\partial_{11} f(x,y,y')\gg 0;\quad \partial_{22} f(x,y,y')\gg 0;
$$
$$
(\partial_{11}f-\partial_{12}f\cdot
\partial_{22}^{-1}f\cdot
\partial_{21}f)(x,y,y')\gg 0; \quad (\partial_{22}f-\partial_{21}f\cdot
\partial_{11}^{-1}f\cdot
\partial_{12}f)(x,y,y')\gg 0
$$
for all ${x\in [a;b]}$. Let's emphasize that, in most eases
(see~\cite{50}) the compact extremum of $\Phi$ is not local one.
\end{remark}

In case of three variables, taking into account the result of
Example~\ref{Exam. 1}, Theorem~\ref{Th 8.8} takes form of

\begin{theorem}\label{th.2.12}
Let a functional ${\Phi:H_1\times H_2\times
H_3\rightarrow\mathbb{R}}$ be twice Fr\'{e}chet differentiable at
a point ${(y_1, y_2, y_3)\in H_1\times H_2\times H_3}$. Suppose
that
\begin{enumerate}
\item [i)] ${\partial_{1}\Phi(y)=0;\,\,\,\,
\partial_{2}\Phi(y)=0;\,\,\,\, \partial_{3}\Phi(y)=0}$;

\item [ii)] ${\partial_{11}\Phi(y)\gg 0}$;\,\,
           ${\partial_{22}\Phi(y)\gg 0}$;\,\,
           ${\partial_{33}\Phi(y)\gg 0}$;

\item [iii)]
          ${\bigtriangleup_{12}^2=(\partial_{22}\Phi-\partial_{21}\Phi\cdot
           \partial_{11}^{-1}\Phi\cdot \partial_{12}\Phi)(y)\gg 0}$; ${\bigtriangleup_{32}^2=(\partial_{22}\Phi-\partial_{23}\Phi\cdot
           \partial_{33}^{-1}\Phi\cdot \partial_{32}\Phi)(y)\gg 0}$;
\\
           ${\bigtriangleup_{13}^3=(\partial_{33}\Phi-\partial_{31}\Phi\cdot
           \partial_{11}^{-1}\Phi\cdot \partial_{13}\Phi)(y)\gg 0}$; ${\bigtriangleup_{23}^3=(\partial_{33}\Phi-\partial_{32}\Phi\cdot
           \partial_{22}^{-1}\Phi\cdot \partial_{23}\Phi)(y)\gg 0}$;

\item [iv)]
${\bigtriangleup_{123}^{23}=\bigtriangleup_{12}^2-\bigtriangleup_{13}^2\cdot
(\bigtriangleup_{13}^3)^{-1}\cdot \bigtriangleup_{12}^3\gg
0}$;\quad
${\bigtriangleup_{132}^{32}=\bigtriangleup_{13}^3-\bigtriangleup_{12}^3\cdot
(\bigtriangleup_{12}^2)^{-1}\cdot \bigtriangleup_{13}^2\gg 0}$;

\item [v)] ${\partial_{11}\Phi(y)\gg(\partial_{12}\Phi\cdot
           \partial_{22}^{-}\Phi\cdot \partial_{21}\Phi+\partial_{13}\Phi\cdot
           \partial_{32}^{-}\Phi\cdot \partial_{21}\Phi+\partial_{12}\Phi\cdot
           \partial_{23}^{-}\Phi\cdot \partial_{31}\Phi+}$\\ ${\quad\quad\quad\quad\quad
           +\partial_{13}\Phi\cdot
           \partial_{33}^{-}\Phi\cdot \partial_{31}\Phi)(y)}$.

Here we denote by $\partial_{ij}^{-}\Phi(y)$ corresponding
elements of the inverse matrix to
$(\partial_{ij}\Phi(y))_{i,j=2}^3$. In this case,

\item [vi)]
          ${\partial_{22}^{-}\Phi(y)=(\partial_{22}^{-1}\Phi\cdot(I_{22}+\partial_{23}\Phi\cdot
           (\bigtriangleup_{23}^3)^{-1}\cdot\partial_{32}\Phi
           \cdot\partial_{22}^{-1}\Phi))(y)}$;
\\
           ${\partial_{33}^{-}\Phi(y)=(\partial_{33}^{-1}\Phi\cdot(I_{33}+\partial_{32}\Phi\cdot
           (\bigtriangleup_{32}^2)^{-1}\cdot\partial_{23}\Phi
           \cdot\partial_{33}^{-1}\Phi))(y)}$;
\\
           ${\partial_{23}^{-}\Phi(y)=-((\bigtriangleup_{32}^2)^{-1}\cdot\partial_{23}\Phi
           \cdot\partial_{33}^{-1}\Phi)(y)}$;
\\
           ${\partial_{32}^{-}\Phi(y)=-((\bigtriangleup_{23}^3)^{-1}\cdot\partial_{32}\Phi
           \cdot\partial_{22}^{-1}\Phi)(y)}$.

\end{enumerate}
Then $\Phi$ has a strong local minimum at ${y}$.
\end{theorem}

Pass to the necessary conditions. In analogous way, remind first a
classical second order necessary condition of a local extremum of
functional acting in Banach space~(\cite{38}, Ch.~I, Th.~8.2.1).
\begin{theorem}\label{oiv.th.2.5}
Let $E$ be a real Banach space and a functional
${\Phi:E\rightarrow\mathbb{R}}$ is twice Fr\'{e}chet
differentiable at a point ${y\in E}$ and has a local minimum at
$y$. Then not only ${\Phi'(y)=0}$ but also ${\Phi''(y)\geq 0}$.
\end{theorem}

Quite analogously with the preceding case, with the help of
Theorems~\ref{Th 8.4} and~\ref{oiv.th.2.5} we obtain the following
second order necessary extremal condition.

\begin{theorem}\label{Th. 8.11}
Let $H_{i}$ ${(i=\overline{1,n})}$ be real separable Hilbert
spaces, ${H=H_{1} \times \ldots \times H_{n}}$, and a functional
${\Phi:H\rightarrow\mathbb{R}}$ is twice Fr\'{e}chet
differentiable and has a local minimum at a point
${y=(y_{1},\ldots,y_{n})\in H}$. Suppose also that the all
operators
$\widetilde{\bigtriangleup_{1}^{1}}(\widetilde{\bigtriangleup_{1}^{2}})^{k}
\mathcal{H}_{n}(\Phi)$ ${(k=\overline{0,n-2})}$ are continuously
invertible. Then not only ${\partial_i\Phi(y)=0}$
${(i=\overline{1,n})}$ but also the inequalities
$$
  \widetilde{\bigtriangleup_{1}^{1}}(\widetilde{\bigtriangleup_{1}^{2}})^{k}
  \mathcal{H}_{n}(\Phi)\geq
  0\,\,\,(k=\overline{0,n-2});\quad
  (\widetilde{\bigtriangleup_{1}^{2}})^{n-1}\mathcal{H}_{n}(\Phi)\geq 0
$$
are valid.
\end{theorem}

In case of three variables Theorem~\ref{Th. 8.11}, using
Example~\ref{Exam. 1.2}, takes form of

\begin{theorem}\label{th.2.12.1}
Let a functional ${\Phi:H_1\times H_2\times
H_3\rightarrow\mathbb{R}}$ is twice Fr\'{e}chet differentiable at
a point ${y\in H_1\times H_2\times H_3}$ and has a local minimum
at $y$. Suppose also that the operators ${\partial_{11}\Phi(y)}$
and
${(\partial_{22}\Phi-\partial_{21}\Phi\cdot\partial_{11}^{-1}\Phi\cdot
\partial_{12}\Phi)(y)}$ are continuously
invertible. Then not only ${\partial_i\Phi(y)=0}$ ${(i=1,2,3)}$
but also the inequalities
\begin{enumerate}
\item [i)]  ${\partial_{11}\Phi(y)\geq 0}$;

\item [ii)] ${(\partial_{22}\Phi-\partial_{21}\Phi\cdot
           \partial_{11}^{-1}\Phi\cdot \partial_{12}\Phi)(y)\geq 0}$;

\item [iii)]
           ${(\partial_{33}\Phi-\partial_{31}\Phi\cdot
           \partial_{11}^{-1}\Phi\cdot \partial_{13}\Phi)(y)-
           (\partial_{32}\Phi-\partial_{31}\Phi\cdot
           \partial_{11}^{-1}\Phi\cdot \partial_{12}\Phi)(y)\times}$ \\
           ${\times(\partial_{22}\Phi-\partial_{21}\Phi\cdot
           \partial_{11}^{-1}\Phi\cdot \partial_{12}\Phi)^{-1}(y)\cdot
           (\partial_{23}\Phi-\partial_{21}\Phi\cdot
           \partial_{11}^{-1}\Phi\cdot \partial_{13}\Phi)(y)\geq 0}$.
\end{enumerate}
are valid.
\end{theorem}

Concluding this item, let's consider an example of finding
functional extremum in case of three variables.

\begin{example}\label{Exam. 2.2}
Let ${H_i=l_2}$ ${(i=1,2,3)}$ and a functional ${\Phi:H_1\times
H_2\times H_3\rightarrow\mathbb{R}}$ has form
$$
\Phi(x,y,z)=\|x\|^2+\|y\|^2+\|z\|^2+(x_1^2+z_1^2+x_1x_2+y_1y_2+z_1z_2+x_1z_1).
$$
Then $\nabla_x\Phi=(4x_1+x_2+z_1, 2x_2+x_1, 2x_3, 2x_4, \ldots)$,
$\nabla_y\Phi=(2y_1+y_2, 2y_2+y_1, 2y_3, 2y_4, \ldots)$,
$\nabla_z\Phi=(4z_1+z_2+x_1, 2z_2+z_1, 2z_3, 2z_4, \ldots)$,
whence $\Phi$ has a critical point at zero and the second partial
derivatives are
$$
\frac{\partial^2 \Phi}{\partial x^2}=\frac{\partial^2
\Phi}{\partial z^2}= \left(\begin{array}{ccccccccc}
4 & 1 & 0 & 0 & \ldots & 0&\ldots\\
1 & 2 & 0 & 0 & \ldots & 0&\ldots\\
0 & 0 & 2 & 0 & \ldots & 0&\ldots\\
\vdots & \vdots & \vdots& \vdots &  & \vdots&\\
0 & 0 & 0 & 0 & \ldots & 2&\ldots\\
\vdots & \vdots & \vdots& \vdots &  & \vdots&
 \end{array} \right),
$$
 \begin{equation}\label{boz_e8.16}
 \frac{\partial^2 \Phi}{\partial y^2}= \left(\begin{array}{ccccccccc}
2 & 1 & 0 &  0 & \ldots & 0&\ldots\\
1 & 2 & 0 &  0 & \ldots & 0&\ldots\\
0 & 0 & 2 &  0 & \ldots & 0&\ldots\\
\vdots &  \vdots & \vdots& \vdots &  & \vdots&\\
0 & 0 & 0 &  0 & \ldots & 2&\ldots\\
\vdots  & \vdots & \vdots& \vdots &  & \vdots&
 \end{array} \right),\,\,\,
 \frac{\partial^2 \Phi}{\partial x \partial z}=
 \frac{\partial^2 \Phi}{\partial z\partial x}=
 \left(\begin{array}{ccccccccc}
1 & 0 &  0 & \ldots & 0&\ldots\\
0 & 0 &  0 & \ldots & 0&\ldots\\
\vdots &  \vdots& \vdots &  & \vdots&\\
0 & 0 & 0  & \ldots & 0&\ldots\\
\vdots & \vdots &  \vdots &  & \vdots&
 \end{array} \right),
\end{equation}
 $$
 \partial^2 \Phi/\partial x \partial y=
 \partial^2 \Phi/\partial y\partial x=
 \partial^2 \Phi/\partial y\partial z=
 \partial^2 \Phi/\partial z\partial x=0.
 $$
 Hence, the Hessian $\mathcal{H}_3(\Phi)$ satisfies conditions of
 Example~\ref{Exam. 2}~(2) and therefore we can restrict oneself
 by testing only five positive definiteness conditions. Note also
 that ${(\partial^2 \Phi/\partial x^2)}$ and ${(\partial^2 \Phi/\partial
 y^2)}$ not commute. So,

 1) It follows immediately from~\eqref{boz_e8.16} that
 $$
\frac{\partial^2 \Phi}{\partial x^2}\gg 0,\quad \frac{\partial^2
\Phi}{\partial y^2}\gg 0, \quad \frac{\partial^2 \Phi}{\partial
z^2}\gg 0.
 $$

 2) The inverse matrix to ${(\partial^2 \Phi/\partial x^2)}$ also
 can be easily calculated from~\eqref{boz_e8.16}:
 $$
\left(\frac{\partial^2 \Phi}{\partial
x^2}\right)^{-1}=\left(\frac{\partial^2 \Phi}{\partial
z^2}\right)^{-1}= \left(\begin{array}{ccccccccc}
2/7 & -1/7 & 0 & 0 & \ldots & 0&\ldots\\
-1/7 & 4/7 & 0 & 0 & \ldots & 0&\ldots\\
0 & 0 & 1/2 & 0 & \ldots & 0&\ldots\\
\vdots & \vdots & \vdots& \vdots &  & \vdots&\\
0 & 0 & 0 & 0 & \ldots & 1/2&\ldots\\
\vdots & \vdots & \vdots& \vdots &  & \vdots&
 \end{array} \right).
$$

3) Hence, a direct calculation shows that the matrix
$$
\frac{\partial^2 \Phi}{\partial x^2}- \frac{\partial^2
\Phi}{\partial x \partial z}\cdot \left(\frac{\partial^2
\Phi}{\partial z^2}\right)^{-1}\cdot\frac{\partial^2
\Phi}{\partial z
\partial x}=
\frac{\partial^2 \Phi}{\partial z^2}- \frac{\partial^2
\Phi}{\partial z \partial x}\cdot \left(\frac{\partial^2
\Phi}{\partial x^2}\right)^{-1}\cdot\frac{\partial^2
\Phi}{\partial x
\partial z}=
 $$
 $$
= \left(\begin{array}{ccccccccc}
26/7 & 1 & 0 & 0 & \ldots & 0&\ldots\\
1 & 2 & 0 & 0 & \ldots & 0&\ldots\\
0 & 0 & 2 & 0 & \ldots & 0&\ldots\\
\vdots & \vdots & \vdots& \vdots &  & \vdots&\\
0 & 0 & 0 & 0 & \ldots & 2&\ldots\\
\vdots & \vdots & \vdots& \vdots &  & \vdots&
 \end{array} \right)
 $$
 is positive definite. So, $\Phi$ has a strong local minimum at
 zero.
\end{example}

\bigskip
\rightline{Department of Algebra and Functional Analysis}
\rightline{Taurida National V.I. Vernadskii University}
\rightline{Vernadskii ave., 4,} \rightline{Simferopol 95007}
\rightline{Ukraine} \rightline{E-mail: old@crimea.edu}
\rightline{katboz@mail.ru}


\begin{thebibliography}{99}
\bibitem{1} \textsc{M.G. Krein}, \textit{The theory of self-adjoint extensions of semi-bounded
Hermitian transformations and its applications}, Mat. Sbornik,
(1947). Part I: \textbf{20} (1947), 431--490; Part II: \textbf{21}
(1947), 366--404. (In Russian).

\bibitem{2} \textsc{Y.L. Shmul’yan}, \textit{Hellinger operator integral},
 Mat. Sbornik, \textbf{49(91)}, No. 4 (1959), 381--430. (In Russian)

\bibitem{4} \textsc{V.P. Potapov}, \textit{The multiplicative structure of J--contractive
matrix functions}, Trudy Moskov. Math. Obsc., \textbf{4 }(1955),
125--236. (In Russian). English translation: Am. Math. Soc.
Transl., \textbf{15}, (1960), 131--243.

\bibitem{11}\textsc{ T.Ya. Azizov and I.S. Iohvidov}, \textit{Theory of  Linear Operators
in Spaces with an Indefinite Metric}, Nauka, Мoskow, 1986. (In
Russian)

\bibitem{40} \textsc{N.D. Kopachevskii, S.G. Krein and Ngo Zuy Kan}, {\it Operator methods
in linear hydrodynamics}, Nauka, Moscow, 1989. (In Russian)

\bibitem{9} \textsc{N.D. Kopachevsky and S. G. Krein},
\textit{Operator approach to linear problems of hydrodynamics},
Vol. 1: Self-adjoint problems for an ideal fluid, Operator Theory
Adv. and Appl., Birkh\"{a}user Verlag, Basel–Boston–Berlin,
\textbf{128} (2001).

\bibitem{10} \textsc{N.D. Kopachevsky and S. G. Krein}, \textit{Operator approach to linear
problems of hydrodynamics}, Vol. 2: Nonself-adjoint problems for
viscous fluids, Operator Theory Adv. and Appl., Birkh\"{a}user
Verlag, Basel–Boston–Berlin, \textbf{146 }(2003).

\bibitem{7} \textsc{J. Mandel}, \textit{On block diagonal and Schur complement
preconditioning},  Numerische Mathematik, Springer, Berlin /
Heidelberg, \textbf{58}, No. 1 (1990), 79--93.


\bibitem{5} \textsc{I. Schur}, \textit{\"{U}ber Potenzreihen, die im Innern des Einheitskreises
beschr\"{a}nkt sind}, J. reine u. angew. Math., (1917). Teil I:
\textbf{147} (1917), 205--232; Teil II: \textbf{148} (1918),
122--145. English translation in I. Gohberg (ed.), I. Schur
methods in operator theory and signal processing, Birkh\"{a}user,
(1986), 31--59 and 36--88.

\bibitem{6} \textsc{A. Bultheel and K. M\"{u}ller},
\textit{On several aspects of $J$--inner functions in Schur
analysis}, Bull. Belg. Math. Soc. Simon Stevin, \textbf{5}, No. 5
(1998), 603--648.

\bibitem{13} \textsc{Y. Arlinskii}, \textit{Conservative realizations of the
functions associated with Schur’s algorithm for the Schur class
operator–valued function}, Operators and Matrices, \textbf{3}, No.
1 (2009), 59-–96.

\bibitem{123}  \textsc{S. Hassi, M. Malamud and H. de Snoo}, \textit{On Krein's extension
theory of nonnegative operators},  Math. Nachr. \textbf{274--275}
(2004), 40--73.

\bibitem{8} \textsc{T.Ya. Azizov and V.A. Khatskevich},
\textit{Bistrict plus-operators and frac- tional linear operator
mappings}, Ukrainian Mathematical Bulletin, \textbf{4}, No. 3
(2007), 311–332. (In Russian)

\bibitem{38} \textsc{A.~Kartan}, {\it Differential calculus. Differential forms}, Mir,
Moscow, 1971. (In Russian)

\bibitem{137} \textsc{I.V. Orlov},
 \textit{Compact extrema: general theory and its applications
to the variational functionals}, Operator Theory: Advances \&
Applications, Birkh\"{a}user Verlag, Basel/Switzerland,
\textbf{190} (2009), 397–417

\bibitem{50}
\textsc{E.V. Bozhonok}, \textit{Compact extremum and compact
analytical properties of basic variational functional in Sobolev
space $W_2^1$}, Manuscript: The dissertations for obtaining
scientific degree of candidate of physical and mathematical
sciences by speciality 01.01.02. Simferopol, 2008.


\end{thebibliography}
\end{document}